\documentclass[12pt]{amsart}     
\usepackage{amssymb,amscd,amsthm,amsmath,amsfonts}
\input xypic
\xyoption{all}

\newcommand{\C}{\ensuremath{\mathbb C}}
\newcommand{\F}{\ensuremath{\mathbb F}}
\newcommand{\N}{\ensuremath{\mathbb N}}
\renewcommand{\P}{\ensuremath{\mathcal P}}
\newcommand{\Q}{\ensuremath{\mathbb Q}}
\newcommand{\V}{\ensuremath{\mathcal V}}
\newcommand{\Z}{\ensuremath{\mathbb Z}}

\newcommand{\diag}{{\operatorname{diag}}}
\newcommand{\rank}{{\operatorname{rank}}}
\newcommand{\tr}{{\operatorname{trace}}}

\newcommand{\ie}{{\em i.e., }}
\newcommand{\eg}{{\em e.g., }}
\newtheorem{theorem}{Theorem}[section] 
\newtheorem{definition}[theorem]{Definition}

\newtheorem{corollary}[theorem]{Corollary}

\newtheorem{lemma}[theorem]{Lemma}

\newtheorem{proposition}[theorem]{Proposition}

{\begin{list}{$\bullet$}{\parsep=0pt \itemsep=0pt 
                         \topsep=6pt \partopsep=\parskip}}%
{\end{list}}

\newcounter{ictr}

\newcounter{nctr}

\begin{document}

\title[$K_0$-theory with $n$-potents]
{$K_0$-theory of $n$-potents in rings and algebras}
\author{Efton Park and Jody Trout}
\address{Box 298900, Texas Christian University, Fort Worth, TX 76129}
\email{e.park@tcu.edu}
\address{6188 Kemeny Hall, Dartmouth College, Hanover, NH 03755}
\email{jody.trout@dartmouth.edu}
\thanks{$2010$ {\it Mathematics Subject Classification}: 18F30, 19A99, 19K99} 

\begin{abstract}
Let $n \geq 2$ be an integer. An \emph{$n$-potent} is an element $e$ of a ring
$R$  such that $e^n = e$.  In this paper, we study $n$-potents in matrices over
$R$ and use them to construct an abelian group $K_0^n(R)$.  If $A$ is a complex
algebra, there is a group isomorphism $K_0^n(A) \cong \bigl(K_0(A)\bigr)^{n-1}$
for all $n \geq 2$. However, for algebras over cyclotomic fields, this is not
true, in general. We consider $K_0^n$ as a covariant functor, and show that it
is also  functorial for a generalization of homomorphism called an
\emph{$n$-homomorphism}. 
\end{abstract}

\maketitle

\section{Introduction}
For more than forty years, $K$-theory has been an essential tool in studying
rings and algebras \cite{Bla98, ROS}.  Given a ring $R$, a simple functorial
object associated to $R$ is the abelian group $K_0(R)$.  There are multiple ways
of defining $K_0(R)$, but the most useful characterization when working with
operator algebras is to define $K_0(R)$ in terms of idempotents (or projections,
if an involution is present) in matrix algebras over $R$; \ie elements $e$ in
$M_k(R)$ for some $k$ with the feature that $e^2 = e$ ($p =p^* = p^2$ in the involutive case).  In this paper, we 
define, for each natural number $n \geq 2$, a group which we denote $K_0^n(R)$. 
This group is constructed from matrices $e$ over $R$ with the property that $e^n
= e$; we call such matrices \emph{$n$-potents}.  We define $K_0^n(R)$ for all
rings, unital or not, and show that $K_0^n$ determines a covariant functor from
rings to abelian groups.  

Let $\Q(n-1)$ be the cyclotomic field obtained from the rationals by adjoining
the $(n-1)$-th roots of unity. We show that $K_0^{n}$ is half-exact on the
subcategory of $\Q(n-1)$-algebras, and given any  such algebra $A$, we show
that $K_0^{n}(A)$ is isomorphic to a direct sum of $n-1$ copies of
$K_0(A)$.   Since a $\C$-algebra $A$ is a  $\Q(n-1)$-algebra for all $n$, 
whatever invariants are contained in $K_0^n(A)$ are already contained in
$K_0(A)$.  However, $K_0^p$ for $p \neq n$ may generate new groups for
cyclotomic algebras, \eg  $K_0^4(\Q(4)) \cong \Z \oplus 2\Z$ (Theorem
\ref{Q(4)_theorem}) which
is not isomorphic to  $K_0^4(\Q(3)) \cong \Z^3$. Thus, $K_0^4$ distinguishes
between the fields $\Q(3)$ and $\Q(4)$, but idempotent, and also tripotent
($n = 3$), $K$-theory does not.

The paper is organized as follows. In Section 2, we define various notions of
equivalence on the set of $n$-potents, and explore the relationships between
these equivalence relations.  Most of our results in this section mirror
analogous facts about idempotents, but in many cases the proofs differ or
are more delicate for $n$-potents.  In Section 3, we define $n$-potent
$K$-theory and study its properties and compute some examples.  Finally, in
Section 4, we consider $n$-homomorphisms on rings and algebras
\cite{Feig1,Feig2,HMM}, and show that $n$-potent $K$-theory is functorial for
such maps; this is a phenomenon that does not appear in ordinary idempotent
$K$-theory. 

The authors thank Dana Williams and Tom Shemanske for their helpful comments and
suggestions.

{\bf Note:}  Unless stated otherwise, all rings and algebras have a unit; \ie a
multiplicative identity, and all ring and algebra homomorphisms are unital.

\section{Equivalence of $n$-potents}

Fix a natural number $n \geq 2$. In this section, we develop the basic theory of $n$-potents, including various equivalence relations among them.  We begin by looking at $n$-potents over general rings, but eventually we will specialize to get a
well-behaved theory. 

\begin{definition}\label{definition of n-potent}
Let $R$ be a ring.  An element $e$ in $R$ is
called an \emph{$n$-potent} if $e^n = e$.  For $n = 2, 3, 4$, we use the terms
\emph{idempotent}, \emph{tripotent}, and \emph{quadripotent}, respectively.  The
set of all $n$-potents in $R$ is denoted $\mathcal{P}^n(R)$.  
\end{definition}

We begin with a very simple but useful fact about $n$-potents:

\begin{lemma}\label{idempotent from n-potent}
Suppose $e$ is an $n$-potent.  Then $e^{n-1}$ is an idempotent.
\end{lemma}

\begin{proof} $
(e^{n-1})^2 = e^{n-1}e^{n-1} = e^ne^{n-2} 
= ee^{n-2} = e^{n-1}. $
\end{proof}

\begin{definition}\label{definition of equivalences}
Let $e$ and $f$ be $n$-potents in a ring $R$.  We say that $e$ and $f$ are 
\emph{algebraically equivalent} and write $e \sim_a f$ if there exist elements
$a$ and $b$ in $R$ such that $e = ab$ and $f = ba$.  We say that $e$ and $f$ are
\emph{similar}
and write $e \sim_s f$ if there exists an invertible element $z$ in $R$ with the
property that $f = zez^{-1}$.  
\end{definition}

\begin{lemma}\label{nice algebraic equivalence}
Suppose that $e$ and $f$ are algebraically equivalent $n$-potents in a ring $R$. 
Then the elements $a$ and $b$ described in Definition \ref{definition of
equivalences} can be chosen so that 
\begin{gather*}
a = e^{n-1}a = af^{n-1} = e^{n-1}af^{n-1}\\
b = f^{n-1}b = be^{n-1} = f^{n-1}be^{n-1}.
\end{gather*}
\end{lemma}

\begin{proof}
Choose elements $\tilde{a}$ and $\tilde{b}$ in $R$ so that 
$\tilde{a}\tilde{b} = e$ and $\tilde{b}\tilde{a} = f$.
Set $a = e^{n-1}\tilde{a}f^{n-1}$ and $b = f^{n-1}\tilde{b}e^{n-1}$. 
Using Lemma \ref{idempotent from n-potent}, we have
\begin{multline*}
ab = (e^{n-1}\tilde{a}f^{n-1})(f^{n-1}\tilde{b}e^{n-1})
= e^{n-1}\tilde{a}f^{n-1}\tilde{b}e^{n-1} \\
= e^{n-1}(\tilde{a}\tilde{b})^ne^{n-1} = 
e^{n-1}e^ne^{n-1} = (e^{n-1})^2e^n = e^{n-1}e = e^n = e.
\end{multline*}
Similarly, $ba = f$.  The two strings of equalities in the statement of the
lemma then follow easily.
\end{proof}

\begin{proposition}\label{equivalence relations}
The relations $\sim_a$ and $\sim_s$ are equivalence relations on
$\mathcal{P}^n(R)$.  
\end{proposition}

\begin{proof}
The only nonobvious point to establish is that $\sim_a$ is transitive.  Let $e$,
$f$, and $g$ be elements of $\mathcal{P}^n(R)$, and suppose that
$e \sim_a f \sim_a g$.  Choose elements $a$, $b$, $c$ and $d$ in $R$ so that
$e = ab$, $f = ba = cd$, and $g = dc$, and set $s = af^{n-2}c$ and 
$t = db$.  Then
\[
st = af^{n-2}cdb = af^{n-1}b = a(ba)^{n-1}b = (ab)^n = e^n = e
\]
and
\[
ts = dbaf^{n-2}c = df^{n-1}c = d(cd)^{n-1}c = (dc)^n = g^n = g.
\]
\end{proof}

\begin{proposition}\label{similarity implies algebraic}
If $e$ and $f$ are similar $n$-potents in a ring $R$, then they are
algebraically equivalent.
\end{proposition}

\begin{proof}
Choose an invertible element $z$ in  $R$ such  that $f = zez^{-1}$, and set
$a = ez^{-1}$ and $b = ze^{n-1}$.  Then $ab = e^n = e$ and $ba = ze^nz^{-1} =
f$.
\end{proof}

As is the case with idempotents, algebraic equivalence does not imply similarity
in general.  However, we do have the following result, just as for idempotents:

\begin{proposition}\label{algebraic stably implies similarity}
Suppose that $e$ and $f$ are algebraically equivalent $n$-potents in a ring $R$. 
Then
\[
\begin{pmatrix} e & 0 \\ 0 & 0 \end{pmatrix} \sim_s
\begin{pmatrix} f & 0 \\ 0 & 0 \end{pmatrix}
\]
in the ring $M_2(R)$ of $2 \times 2$ matrices over $R$.
\end{proposition}

\begin{proof}
Choose elements $a$ and $b$ in $R$ so that $e = ab$ and $f = ba$; without loss
of generality, we assume that $a$ and $b$ satisfy the conclusions of
Lemma \ref{nice algebraic equivalence}.  Define
\[
u = \begin{pmatrix} 1 - f^{n-1} & b \\ af^{n-2} & 1 - e^{n-1} \end{pmatrix}
\]
and
\[
v = \begin{pmatrix} 1 - e^{n-1} & e^{n-1} \\ e^{n-1} & 1 - e^{n-1}
\end{pmatrix}.
\]
Straightforward computation yields that both $u^2$ and $v^2$ equal the identity
matrix in $M_2(R)$, and thus each is its own inverse.   Set $z = uv$.  Then we compute that
\begin{align*}
z\begin{pmatrix} e & 0 \\ 0 & 0 \end{pmatrix}z^{-1} &=
\begin{pmatrix} 1 - f^{n-1} & b \\ af^{n-2} & 1 - e^{n-1} \end{pmatrix} 
\begin{pmatrix} 0 & 0 \\ 0 & e \end{pmatrix}
\begin{pmatrix} 1 - f^{n-1} & b \\ af^{n-2} & 1 - e^{n-1} \end{pmatrix}  \\
&=\begin{pmatrix} beaf^{n-2} & 0 \\ 0 & 0 \end{pmatrix} = \begin{pmatrix} f & 0 \\ 0 & 0 \end{pmatrix}
\end{align*}
since $beaf^{n-2} = b(ab)a(ba)^{n-2} = (ba)^n = f^n = f.$
\end{proof}

\begin{definition}\label{definition of orthogonal}
We say $n$-potents $e$ and $f$ in a ring $R$ are \emph{orthogonal} if 
$ef = fe = 0$, in which case we write $e \perp f$.
\end{definition}

\noindent The next result follows immediately by mathematical induction.

\begin{proposition}\label{add orthogonal}
Let $e$ and $f$ be orthogonal $n$-potents in a ring $R$.  Then 
$(e + f)^k = e^k + f^k$.  In particular, $e + f$ is an $n$-potent.
\end{proposition}


\begin{proposition}\label{addition of algebraic}
For $i = 1, 2$, let $e_i$ and $f_i$ be algebraically equivalent $n$-potents in a
ring $R$.  Suppose that $e_1$ and $f_1$ are orthogonal to $e_2$ and $f_2$,
respectively.  Then $e_1 + e_2$ and $f_1 + f_2$ are algebraically equivalent. 
\end{proposition}

\begin{proof} For $i = 1, 2$, choose $a_i$ and $b_i$ so that $e_i = a_ib_i$, 
$f_i = b_ia_i$, and so that $a_i$ and $b_i$ satisfy the conclusion of Lemma
\ref{nice algebraic equivalence}.  Then
\[
a_1b_2 = a_1f_1^{n-1}f_2^{n-1}b_2 = 0.
\]
Similarly, $b_2a_1$, $a_2b_1$, and $b_1a_2$ are also zero.  Thus
\[
(a_1 + a_2)(b_1 + b_2) = a_1b_1 + a_2b_2 = e_1 + e_2
\]
and
\[
(b_1 + b_2)(a_1 + a_2) = b_1a_1 + b_2a_2 = f_1 + f_2,
\]
whence $e_1 + e_2$ is algebraically equivalent to $f_1 + f_2$.  
\end{proof}

\begin{proposition}\label{stable equivalence}
Let $e$ and $f$ be $n$-potents in a ring $R$.  
\begin{enumerate}
\item[(a)] $\begin{pmatrix} e & 0 \\ 0 & f \end{pmatrix} \sim_a
\begin{pmatrix} f & 0 \\ 0 & e \end{pmatrix}$ and 
$\begin{pmatrix} e & 0 \\ 0 & 0 \end{pmatrix} \sim_a
\begin{pmatrix} 0 & 0 \\ 0 & e \end{pmatrix}$.
\item[(b)] If $e \perp f$ then $\begin{pmatrix} e & 0 \\ 0 & f \end{pmatrix}
\sim_a  \begin{pmatrix} e+ f & 0 \\ 0 & 0 \end{pmatrix}$.
\end{enumerate}
\end{proposition}

\begin{proof}
Define
\[
a = \begin{pmatrix} 0 & e \\ f & 0 \end{pmatrix}\quad\text{and}\quad
b = \begin{pmatrix} 0 & f^{n-1} \\ e^{n-1} & 0 \end{pmatrix}
\]
in $M_2(R)$.  Then
\[
ab = \begin{pmatrix} 0 & e \\ f & 0 \end{pmatrix}
\begin{pmatrix} 0 & f^{n-1} \\ e^{n-1} & 0 \end{pmatrix} =
\begin{pmatrix} e^n & 0 \\ 0 & f^n \end{pmatrix}
= \begin{pmatrix} e & 0 \\ 0 & f \end{pmatrix}
\]
and
\[
ba = \begin{pmatrix} 0 & f^{n-1} \\ e^{n-1} & 0 \end{pmatrix}
\begin{pmatrix}0 & e \\ f & 0 \end{pmatrix} =
\begin{pmatrix} f^n & 0  \\ 0 & e^n \end{pmatrix}
= \begin{pmatrix} f & 0 \\ 0 & e \end{pmatrix},
\]
which establishes the first part of (a); to obtain the second part, simply take
$f$ to be zero.

To prove (b), first observe that if $e \perp f$, then $e + f$ is an $n$-potent
by Proposition \ref{add orthogonal}.  Define
\[
a = \begin{pmatrix} e & 0 \\ f & 0 \end{pmatrix} \quad\text{and}\quad
b = \begin{pmatrix} e^{n-1} & f^{n-1} \\ 0 & 0 \end{pmatrix}.
\]
Then
\[
ab = \begin{pmatrix} e & 0 \\ f & 0 \end{pmatrix}
\begin{pmatrix} e^{n-1} & f^{n-1} \\ 0 & 0 \end{pmatrix} =
\begin{pmatrix} e^n & ef^{n-1} \\ fe^{n-1} & f^n \end{pmatrix} =
\begin{pmatrix} e & 0 \\ 0 & f \end{pmatrix}
\]
and
\[
ba = \begin{pmatrix} e^{n-1} & f^{n-1} \\ 0 & 0 \end{pmatrix}
\begin{pmatrix} e & 0 \\ f & 0 \end{pmatrix} = 
\begin{pmatrix} e^n + f^n & 0  \\ 0 & 0 \end{pmatrix} = 
\begin{pmatrix} e + f & 0 \\ 0 & 0 \end{pmatrix},
\]
whence the result follows.
\end{proof}

Later in this paper we will restrict our attention to $n$-potent $K$-theory of {\em cyclotomic algebras}:

\begin{definition}\label{definition of cyclotomic algebra}
For each integer $n \geq 2$, the {\em cyclotomic field} $\Q(n-1)$ is the
field obtained by adjoining the $(n-1)$st primitive root of unity $\zeta_{n-1} = e^{2\pi i / (n-1)}$
 to the field $\Q$ of rational numbers.  A {\em cyclotomic algebra} is a $\Q(n-1)$-algebra for some $n \geq 2$.
\end{definition}

Observe that $\Q(n-1) \subset \C$, and therefore every $\C$-algebra is canonically a $\Q(n-1)$-algebra for all
$n$. 

\begin{definition}\label{definition of n-partition}
Let $\F$ be a field and let $A$ be an $\F$-algebra with unit.
 An {\em $n$-partition of unity} is an ordered $n$-tuple $(e_0, e_1, \dots,
e_{n-1})$ of idempotents in $A$ such that 
\begin{enumerate}
\item $e_0 + e_1 + \cdots + e_{n-1} = 1$;
\item  $e_0, e_1, \dots, e_{n-1}$ are pairwise orthogonal;
 \ie $e_j e_k = \delta_{jk} e_k$ for all $0 \leq j, k \leq n-1$.
\end{enumerate}
\end{definition}

Note that $e_0 = 1 - (e_1 + \cdots + e_{n-1})$ is completely determined by $e_1, 
e_2, \dots, e_{n-1}$ and is thus redundant in the notation for an $n$-partition
of unity. 

Cyclotomic algebras admit a distinguished $n$-partition
of unity.  Set $\omega_ 0 = 0$ and let $\omega_k = e^{2 \pi i(k-1) / ( n-1)}$ for
$1\leq k \leq n-1$.  Note that 
$\omega_1, \dots, \omega_{n-1}$ are the $(n-1)$st roots of unity,
 and $\Omega_n = \{\omega_0, \omega_1, \dots, \omega_{n-1}\}$ is the set
 of roots of the polynomial equation $x_n - x = 0$.

\begin{theorem}\label{partition of n-potents}
Let $A$ be a $\Q(n-1)$-algebra with unit, and suppose $e$ is an $n$-potent in
$A$.  Then there exists a unique $n$-partition of unity
$(e_0, e_1, \dots, e_{n-1})$ in $A$ such that
\[
e = \sum_{k = 1}^{n-1} \omega_k e_k.
\]
\end{theorem}

\begin{proof}
Let $p_0, p_1, \dots, p_{n-1} \in \Q(n-1)[x]$ be the Lagrange polynomials   \[
p_k(x) = \frac{\prod_{j \neq k} (x - \omega_j)}
{\prod_{j \neq k} (\omega_k - \omega_j)}.
\]
In particular, $p_0(x) = 1 - x^{n-1}$.  Each polynomial $p_k$ has degree $n-1$
and satisfies $p_k(\omega_k) = 1$ and $p_k(\omega_j) = 0$ for all $j \neq k$.
We claim that for all numbers $\alpha \in \Q(n-1) \subseteq \C$, 
\begin{equation}\label{sumone}
\sum_{k = 0}^{n-1} p_k(\alpha ) = p_0(\alpha ) + \cdots + p_{n-1}(\alpha ) = 1
\end{equation}
and that
\begin{equation}\label{identity}
\alpha  = \sum_{k = 0}^{n-1} \omega_k p_k(\alpha ).
\end{equation}
Indeed, these identities follow from the fact that these polynomial equations
have degree $n-1$ but are satisfied by the $n$ distinct points in $\Omega_n$.

Now, given any $\omega_i^n = \omega_i$ in $\Omega_n$ it follows that $p_k(\omega_i)^2 = p_k(\omega_i)$.
Hence, for any $n$-potent $e \in A$, if we define $e_k = p_k(e)$, then each
$e_k$ is  an idempotent in $A$, and Equation (\ref{sumone}) implies that
\[
\sum_{k=0}^{n-1} e_k = \sum_{k=0}^{n-1} p_k(e) = 1.
\]
These idempotents are pairwise orthogonal, because 
\[
e_j e_k = p_j(e)p_k(e) = 0
\]
for $j \neq k$.  Finally,
\[
e = \sum_{k = 1}^{n-1} \omega_k p_k(e) = \sum_{k =1}^{n-1} \omega_k e_k
\]
by Equation (\ref{identity}).
\end{proof}

\section{$K_0$-theory with $n$-potents}

We can now proceed to construct our $n$-potent $K$-theory groups.

\begin{definition}\label{definition of M(R)}
Let $R$ be a ring.  For all $k \geq 1$, let $\P^n_k(R)$ 
denote the set of $n$-potents in $M_k(R)$, and let 
$i_k$ denote the inclusion
\[
i_k(a) = \begin{pmatrix} a & 0 \\ 0 & 0 \end{pmatrix}
\]
of $M_k(R)$ into $M_{k+1}(R)$, as well as its restriction as a map from 
$\P^n_k(R)$ to $\P^n_{k+1}(R)$.  
Define $M_\infty(R)$ and $\P^n_\infty(R)$ to be the (algebraic)
direct limits
\[
M_\infty(R) = \bigcup_{k=1}^\infty M_k(R), \quad \P^n_\infty(R) =
\bigcup_{k=1}^\infty \P^n_k(R) = \P^n(M_\infty(R)).
\]
\end{definition}

We define a binary operation $\oplus$ on $\P^n_\infty(R)$ as follows: let $e$ and 
$f$ be elements of $\P^n_\infty(R)$,  choose the smallest natural numbers
$k$ and $\ell$ such that $e \in M_k(R)$ and $f \in M_l(R)$, and set
\[
e \oplus f = \diag(e, f) = \begin{pmatrix} e & 0 \\ 0 & f \end{pmatrix} \in
\P^n_{k+l}(R) \subset \P^n_\infty(R).
\]

\begin{definition}\label{definition of V(R)}
Let $R$ be a ring, and define an equivalence relation $\sim$ on 
$\P^n_\infty(R)$ as follows:  take $e$ and $f$ in $\P^n_\infty(R)$, and choose a
natural number $k$ sufficiently large that $e$ and $f$ are elements of
$\P^n_k(R)$.
Then $e \sim f$ if $e \sim_a f$ in $M_k(R)$.  We let $\V^n(R)$ denote the set of 
equivalence classes of $\sim$.
\end{definition}

Note that if $e = ab$ and $f = ba$ in $M_k(R)$, then 
\[
\begin{pmatrix} e & 0 \\ 0 & 0 \end{pmatrix} =
\begin{pmatrix} a & 0 \\ 0 & 0 \end{pmatrix}
\begin{pmatrix} b & 0 \\ 0 & 0 \end{pmatrix}
\]
and
\[
\begin{pmatrix} f & 0 \\ 0 & 0 \end{pmatrix} =
\begin{pmatrix} b & 0 \\ 0 & 0 \end{pmatrix}
\begin{pmatrix} a & 0 \\ 0 & 0 \end{pmatrix},
\]
and therefore the equivalence relation described in Definition 
\ref{definition of V(R)} is well-defined.

Note that for any $n$-potent $e, f$ in $M_\infty(R)$, we get
\[
e = \begin{pmatrix} e & 0 \\ 0 & 0 \end{pmatrix} \perp \begin{pmatrix} 0 & 0
\\ 0 & f \end{pmatrix} \sim \begin{pmatrix} f & 0 \\ 0 & 0 \end{pmatrix} = f.
\]
Thus, the binary operation $\oplus$ induces a binary operation $+$ on
$\V^n_\infty(R)$ as follows: take $e$ and $f$ in $\P^n_\infty(R)$, and define
\[
[e] + [f] = [e \oplus f] = \left[ \begin{pmatrix} e & 0 \\ 0 & f
\end{pmatrix}\right].
\]
This operation is well-defined and commutative by Propositions
\ref{add orthogonal} and \ref{stable equivalence}.

The next proposition is straightforward and left to the reader.

\begin{proposition}\label{V abelian monoid}
For every ring $R$ and natural number $n \geq 2$, $\V^n(R)$ is an abelian
monoid under the addition defined above, and whose identity element is the
class of the zero $n$-potent.  If $\alpha : R \longrightarrow S$ is a unital
ring homomorphism, then the induced map 
$\V^n(\alpha) : \V^n(R) \longrightarrow \V^n(S)$ given by
\[
\V^n(\alpha)([(a_{ij})]) = [(\alpha(a_{ij}))]
\]
is a well-defined homomorphism of abelian semigroups. The correspondences
$R \mapsto \V^n(R)$ and $\alpha \mapsto \V^n(\alpha)$ induce a covariant functor from the category of rings and
ring homomorphisms to the category of abelian monoids and monoid 
homomorphisms.
\end{proposition}

\begin{definition}\label{definition of K0}
Let $R$ be a ring and let $n \geq 2$ be a natural number.  We define
$K_0^n(R)$ to be the Grothendieck completion \cite{RLL} of the abelian
monoid $\V^n(R)$.  Given an $n$-potent $e$ in $\P^n_\infty(R)$, we denote
its class in $K_0^n(R)$ by $[e]$.
\end{definition}

In light of Propositions \ref{similarity implies algebraic} and
\ref{algebraic stably implies similarity}, we could have alternatively used 
similarity to define $\V^n(R)$, and hence $K_0^n(R)$. 

\begin{proposition}\label{K0 is a functor}
The assignments $R \mapsto K_0^n(R)$ determines a covariant functor
from the category of rings and ring homomorphisms to the category of abelian
groups and group homomorphisms.
\end{proposition}

\begin{proof}
Proposition \ref{V abelian monoid} states that $\V$ is a covariant functor from
the category of rings to the category of abelian monoids, and Grothendieck
completion determines a covariant functor from the category of abelian monoids
to the category of abelian groups; we get the desired result by composing these
two functors.
\end{proof}

The following result shows that for (unital) algebras over a field of
characteristic $\neq 2$, the tripotent $K$-theory functor $K_0^3$ offers us no
new invariants over ordinary idempotent $K$-theory. However, we will see later
(Theorem \ref{Q(4)_theorem}) that the situation is subtly different for $K_0^4$.

\begin{theorem}\label{tripotent_thm}
Let $\F$ be a field with characteristic $\neq 2$. If $A$ is a unital algebra
over $\F$ then there is a natural isomorphism 
\[
K_0^3(A) \cong \bigl( K_0(A) \bigr)^2
\] of abelian groups.
\end{theorem}

\begin{proof} If $e = e^3 \in M_\infty(A)$ is a tripotent, then one can easily
check that
\[ e_1 = \frac12 (e^2 +e) \quad \text{ and } \quad e_2 = \frac12(e^2 -e) \]
are (unique) idempotents in $M_\infty(A)$ such that $e = e_1 - e_2$. It follows
that we have a natural bijection of abelain monoids
\[ \begin{aligned}
\V^3(A) & \to \V^2(A) \oplus \V^2(A)\\
[e] & \mapsto [e_1] \oplus [e_2]
\end{aligned}\]
with inverse map $[e_1] \oplus [e_2] \mapsto [e_1 \oplus -e_2]$. Since these
maps are additive, the result easily follows. 
\end{proof}

While $K_0^n(R)$ is well-defined for any ring $R$, to obtain a well-behaved
theory where the usual exact sequences exist, we must restrict our attention to
a smaller class of rings. The problem is that unlike the situation for
idempotents, it is not generally true that if $e$ is an $n$-potent, then so is
$1 - e$.  However, given an $n$-potent in an algebra over the cyclotomic field
$\Q(n-1)$, there is an adequate substitute:

\begin{definition}\label{definition of complementary n-potent}
Let $e$ be an $n$-potent in a $\Q(n-1)$-algebra $A$, and write
\[
e = \sum_{k = 1}^{n-1} \omega_k e_k
\]
as in the conclusion of Theorem \ref{partition of n-potents}.  We define an
$n$-potent
\[
e^\perp = \sum_{k = 1}^{n-1}\diag\bigl(\omega_1(1 - e_1), \omega_2(1 - e_2),
\dots,
\omega_{n-1}(1 - e_{n-1})\bigr) \in M_{n-1}(A)
\]
and call $e^\perp$ the \emph{complementary $n$-potent} of $e$.
\end{definition}

Observe that if $n=2$, this definition agrees with the usual one for
idempotents; \ie $e^\perp = 1 - e$. Note also that $e \oplus e^\perp \sim_s
\omega$, where
\[
\omega = \diag(\omega_1 1_A, \dots, \omega_n 1_A) \in M_{n-1}(\Q(n-1)) \subseteq
M_{n-1}(A).
\] 

\begin{proposition}[Standard Picture of $K_0^n(A)$]
\label{standard picture of K0}
Let $n \geq 2$ be a natural number and let $A$ be a $\Q(n-1)$-algebra. Then
every element of $K_0^n(A)$ can be written in the form $[e] - [\omega]$, where
$e$ in an
$n$-potent in $M_k(A)$ for some natural number $k$ and $\omega$ is a diagonal 
$n$-potent in $M_k(\Q(n-1))$.
\end{proposition}

\begin{proof}
Start with an element $[\tilde{e}] - [\tilde{f}]$ in $K_0^n(A)$, and take
$\tilde{f}^\perp$ to be the complementary $n$-potent of $f$ as defined in
Definition \ref{definition of complementary n-potent}.  Then 
\[
[\tilde{e}] - [\tilde{f}] = \bigl([\tilde{e}] + [\tilde{f}^\perp]\bigr) 
- \bigl([\tilde{f}] + [\tilde{f}^\perp]\bigr).
\]
The $n$-potents $\tilde{f}$ and $\tilde{f}^\perp$ are orthogonal, and therefore
\[
[\tilde{f}] + [\tilde{f}^\perp] = [\tilde{f} + \tilde{f}^\perp] = [\omega],
\]
where $\omega$ has the desired form.  Finally we take 
$e$ to be $\tilde{e} \oplus \tilde{f}^\perp$, and by enlarging the matrix
$\omega$, we obtain the desired result.
\end{proof}

\begin{proposition}\label{add trivial}
Let $n \geq 2$ and let $A$ be a $\Q(n-1)$-algebra. Suppose $e$ and $f$ are
$n$-potents in $M_\infty(A)$.  Then $[e] = [f]$ in $K_0^n(A)$ if and only if $e
\oplus \omega$ is similar to $f \oplus \omega$ for some $n$-potent 
$\omega$ in $M_\infty(\Q(n-1))$.
\end{proposition}

\begin{proof}
The \lq\lq only if\rq\rq\  direction is obvious.  To show the inference in the
opposite direction, suppose that $[e] = [f]$ in $K_0^n(A)$.  By the definition
of the Grothendieck completion, $e \oplus \tilde{e}$ is similar to $f \oplus
\tilde{e}$ for some $n$-potent $\tilde{e}$ in $M_\infty(A)$. 
Then $e \oplus \tilde{e} \oplus \tilde{e}^\perp$ is similar to
$f \oplus \tilde{e} \oplus \tilde{e}^\perp$.
But if we write $\tilde{e} = \sum_{k=1}^{n-1}\omega_k\tilde{e}_k$  as in Theorem
\ref{partition of n-potents}, then Proposition \ref{stable equivalence}(b)
implies that
\[
\tilde{e} \sim_s \diag\bigl(\omega_1\tilde{e}_1, \omega_2\tilde{e}_2, \dots, 
\omega_{n-1}\tilde{e}_{n-1}\bigr).
\]
Therefore $\tilde{e} \oplus \tilde{e}^\perp$ is similar to an $n$-potent
in $M_\infty(\Q(n-1))$, and the proposition follows.
\end{proof}

We next turn our attention to $n$-potent $K$-theory for nonunital
algebras.  Given a nonunital $\Q(n-1)$-algebra $A$, we define its
\emph{unitization} $A^+$ as the unital $\Q(n-1)$-algebra
$A^+ = \{(a, \lambda): a \in A, \lambda \in \Q(n-1) \}$, where addition and
scalar multiplication are defined componentwise, and multiplication is given by
$(a, \lambda)(b, \tau) = (ab + a\tau + b\lambda, \lambda\tau)$.  

\begin{definition}\label{definition of nonunital K0}
Let $A$ be a nonunital $\Q(n-1)$-algebra, and let $A^+$ be its unitization.  Let 
$\pi: A^+ \longrightarrow \Q(n-1)$ be the algebra homomorphism 
$\pi(a, \lambda) = \lambda$.  Then we define $K_0^n(A) = \ker\pi_*$.
\end{definition}

It is easy to see that $\pi_*$ is surjective, so by definition of $K_0^n(A)$ we
have a
short exact sequence
\[
\xymatrix{
0 \ar[r] &  K_0^n(A) \ar[r] &  K_0^n(A^+) \ar[r]^-{\pi_*} &   K_0^n(\Q(n-1)) 
\ar[r] \ar@/^/[l]^-{\psi_*} & 0
}
\]
with splitting induced by the map $\psi: \Q(n-1) \longrightarrow A^+$ defined by
$\psi(\lambda) = (0, \lambda)$.  In addition, it is easy to check that if $A$
already has a unit and we form $A^+$, then $\ker \pi_*$ is naturally isomorphic
to our original definition of $K_0^n(A)$.

\begin{proposition}\label{standard picture of nonunital K0}
Let $A$ be a nonunital $\Q(n-1)$-algebra.  Then every element of 
$K_0^n(A)$ can be written in the form $[e] - [s(e)]$, where $e$ is an
$n$-potents in $M_k(A^+)$ for some integer $k \geq 1$, and $s = \psi \circ \pi 
: A^+ \to A^+$ is
the scalar mapping \cite[Sect. 4.2.1]{RLL}.
\end{proposition}

\begin{proof}
Follows directly from Proposition \ref{standard picture of K0} and
Definition \ref{definition of nonunital K0}.
\end{proof}

\begin{proposition}[Half-exactness]\label{half-exactness}
Every short exact sequence 
\[
\xymatrix{
0 \ar[r] & I \ar[r]^-{i} & A \ar[r]^-{q} & A/I \ar[r] & 0 }
\]
of $\Q(n-1)$-algebras, with $A$ unital,  induces an exact sequence
\[
\xymatrix{
K_0^n(I) \ar[r]^-{i_*} & K_0^n(A) \ar[r]^-{q_*} & K_0^n(A\slash I)}
\]
of abelian $n$-potent $K$-theory groups.
\end{proposition}

\begin{proof}
Since $q \circ i = 0$, we have by functoriality that $q_* \circ i_* = 0$
and so the image of $K_0^n(I)$ under $i_*$ in $K_0^n(A)$ is contained in the
kernel of 
$q_*$.  To show the reverse inclusion, suppose we have 
$[e] - [\lambda]$ in $K_0^n(A)$ such  that $q_*\bigl([e] - [\lambda]\bigr) = 0$.
Then $[q(e)] = [q(\lambda)] = [\lambda]$ in $K_0^n(A\slash I)$.
By Proposition \ref{add trivial}, there exists an $n$-potent $\tau$ in 
$M_\infty(\Q(n-1))$ so that
\[
q(e) \oplus \tau \sim_s \lambda \oplus \tau.
\]
Choose $N$ sufficiently large so that we may view $e$, $\lambda$, and $\tau$ as
$N$ by $N$ matrices, and choose $z$ in $GL_{2N}(A\slash I)$ so that
\[
z\bigl( q(e) \oplus \tau\bigr)z^{-1} = \lambda \oplus \tau.
\]
By Proposition 3.4.2 and Corollary 3.4.4 in \cite{Bla98}, we can lift 
$\diag(z, z^{-1})$ to an element $u$ in $GL_{4N}(A)$.  Set
$f = u(e \oplus \tau)u^{-1}$.  Then
\[
q(f) = \diag(z, z^{-1})(q(e) \oplus \tau)\diag(z^{-1}, z)
= \lambda \oplus \tau,
\]
and thus $f$ and $\lambda \oplus \tau$ are in $M_{4N}(I^+)$.  Therefore
\[
[e] - [\lambda] = [e \oplus \tau] - [\lambda \oplus \tau] 
=  i_*([f] - [\lambda \oplus \tau])
\]
is in the image of $K_0^n(I)$ under $i_*$ as desired.
\end{proof}

Note that our proof of Proposition \ref{half-exactness} relies critically on Proposition 
\ref{add trivial}, which in turn is proved using the standard picture of $K_0^n(A)$.
We do not have a standard picture for $K_0^k(A)$ when $k \neq n$, and it seems likely to the
authors that $K_0^k$ is, in fact, not half-exact in this case.  However, 
we do not have a counterexample where half-exactness fails to hold.

While it is not at all obvious from its definition, $K_0^n(A)$ can be identified
with a more familiar object.

\begin{theorem}\label{iso}
Let $n \geq 2$ be a natural number and let $A$ be a not necessarily unital
$\Q(n-1)$-algebra.  Then there is a natural isomorphism
\[
K_0^n(A) \cong \bigl( K_0(A)\bigr)^{n-1}
\] of abelian groups.
\end{theorem}

\begin{proof}
First consider the case where $A$ is unital.   We define a homomorphism 
$\tilde{\psi}: \mathcal{V}^n(A) \longrightarrow
\bigl(\mathcal{V}_0(A)\bigr)^{n-1}$ in the following way:
for each $n$-potent $e = \sum_{k=1}^{n-1}\omega_k e_k$ in $M_\infty(A)$, set
\[
\tilde{\psi} [e] = \bigl([e_1], [e_2], \dots, [e_{n-1}]\bigr).
\]
It is easy to check that $\tilde{\psi}$ is additive and well-defined.
Next, define a homomorphism
$\tilde{\phi}: \bigl(\mathcal{V}(A)\bigr)^{n-1} \longrightarrow
\mathcal{V}^n(A)$ 
by the formula
\begin{multline*}
\tilde{\phi}\bigl([f_1], [f_2], \dots, [f_{n-1}]\bigr) = \\
\bigl[ \omega_1\diag(f_1, 0, 0, \dots, 0) + 
 \omega_2\diag(0, f_2, 0, \dots, 0) + \cdots \\
  + \omega_{n-1}\diag(0, 0, \dots, 0, f_{n-1})\bigr].
\end{multline*}
Note that
\begin{align*}
&\tilde{\psi}\tilde{\phi}\bigl( [f_1], [f_2], \dots, [f_{n-1}]\bigr)\\
&{\hskip 24pt} = 
\psi\bigl[ \omega_1\diag(f_1, 0, \dots, 0) + 
\cdots + \omega_{n-1}\diag(0, 0, \dots, f_{n-1})\bigr] \\
&{\hskip 24pt}= \bigl( [\diag(f_1, 0, \dots, 0)], [\diag(0, f_2, \dots, 0)]
\dots
 [\diag(0,0, \dots, f_{n-1})]\bigr) \\
&{\hskip 24pt}= \bigl( [f_1], [f_2], \dots, [f_{n-1}]\bigr)
\end{align*}
and 
\begin{align*}
\tilde{\phi}\tilde{\psi} [e] &=
\phi\bigl( [e_1], [e_2], \dots, [e_{n-1}]\bigr) \\
&= \bigl[ \omega_1\diag(e_1, 0, \dots, 0) 
+ \cdots + \omega_{n-1}\diag(0, 0, \dots, e_{n-1})\bigr] \\
&= \bigl[\diag(\omega_1 e_1, \omega_2 e_2 , \dots, \omega_{n-1} e_{n-1})\bigr]
\\
& = [e],
\end{align*}
where the last equality is a consequence of Proposition \ref{stable
equivalence}(b).
The universal mapping property of the Grothendieck completion implies that
$\tilde{\psi}$  extends uniquely to an abelian group isomorphism
\[ \psi: K^n_0(A) \longrightarrow \bigl(K_0(A)\bigr)^{n-1},\] and thus
the theorem is true for unital $\Q(n-1)$-algebras. 

Now suppose that $A$ does not have a unit.  Then we have the following
commutative diagram with exact rows:
\[
\xymatrix{
0 \ar[r] & K_0^n(A) \ar[r] & K_0^n(A^+) \ar[r] \ar[d]^-{\cong} & K_0^n(\Q(n-1))
\ar[r] \ar[d]^-{\cong} & 0\\
0 \ar[r] & K_0(A)^{n-1} \ar[r] & K_0(A^+)^{n-1} \ar[r] & K_0(\Q(n-1))^{n-1}
\ar[r] & 0}
\]

An easy diagram chase shows that there is a unique group isomorphism
from $K_0^n(A)$ to $\bigl( K_0(A)\bigr)^{n-1}$ that makes 
the diagram commute.
\end{proof}

Since a complex algebra is a $\Q(n-1)$-algebra for all values of $n$, we have
the following immediate corollary.

\begin{corollary}\label{complex_isom} If $A$ is a $\C$-algebra, there are
natural isomorphisms
\[K_0^n(A) \cong\bigl(K_0(A)\bigr)^{n-1}\] of abelian groups for all natural
numbers $n \geq 2$.
\end{corollary}

We now arrive at the result that suggests why we should consider all
$K_0^n$-functors for algebras over a cyclotomic field.

\begin{theorem}\label{Q(4)_theorem}
Let $\Q(4) = \Q[i]$ be the $4$th cyclotomic field. Then we have the following
isomorphisms of abelian groups:
\[
\begin{aligned}
K_0^2(\Q(4)) & \cong \Z, \\
K_0^3(\Q(4)) & \cong \Z^2, \\
K_0^4(\Q(4)) & \cong \Z \oplus 2 \Z, \\
K_0^5(\Q(4)) & \cong \Z^4.
\end{aligned}
\]
Thus, $K_0^4(\Q(4)) \not \cong \Z^3 \cong K_0^4(\Q(3))$.
\end{theorem}

\begin{proof} Since $\Q(4)$ is a field \cite{ROS}, we have $K_0^2(\Q(4)) =
K_0(\Q(4)) \cong \Z$. The field $\Q(4)$ has characteristic $0 \neq 2$, so
Theorem \ref{tripotent_thm} implies that $K_0^3(\Q(4)) \cong
\bigl(K_0(\Q(4)\bigr)^2 \cong \Z^2$. Theorem \ref{iso} implies that we have an
isomorphism $K_0^5(\Q(4)) \cong \bigl(K_0(\Q(4)\bigr)^4 \cong \Z^4$.

However, the spectrum of $4$-potents is contained in 
\[
\Omega_3 = \Big\{0, 1, -\frac12+\frac{\sqrt{3}}2 i, -\frac12+\frac{\sqrt{3}}2
i\Big\} \subset \C
\]
which is {\bf not} contained in $\Q(4)$ since the two primitive $3$rd roots of
unity $\omega = \zeta_3 = -\frac12+\frac{\sqrt{3}}2 i$ and $\bar{\omega} =
\bar{\zeta}_3 = -\frac12-\frac{\sqrt{3}}2 i$ are not in $\Q(4)= \Q[i]$.

Given any $4$-potent $e \in M_n(\Q(4)) \subset M_n(\C)$ we can uniquely write
\[
e = e_1 + \omega e_2 + \bar{\omega} e_3,
\]
where $e_1, e_2, e_3$ are orthogonal idempotents in $M_n(\C)$ that sum to an
idempotent $e_1 + e_2 + e_3 = e^3$ in $M_n(\Q(4))$ by Lemma \ref{idempotent from
n-potent}. We thus have that
\[
\begin{aligned}
e^2 & = e_1 + \bar{\omega}e_2 + \omega e_3 \\
e^3 & = e_1 + e_2 + e_3 
\end{aligned}
\]
because $\omega^2 = \bar{\omega}, \bar{\omega}^2 = \omega$, and $\omega^3 =
\bar{\omega}^3 = 1$. Since $\omega + \bar{\omega} = -1$, this implies that the
first idempotent 
\[
e_1 = (e + e^2 + e^3)/3 \in M_n(\Q(4))
\]
and the {\it sum} of the last two idempotents
\[
e_2 + e_3 = e^3 - e_1 \in M_n(\Q(4))
\]
are both in $M_n(\Q(4))$. Using a simple trace argument and the fact that
$\omega, \bar{\omega} \not \in \Q(4)$, we conclude that
\[
\rank(e_2) = \tr(e_2) = \tr(e_3) = \rank(e_3),
\]
and so $\rank(e_2 + e_3) = \tr(e_2 + e_3) = 2 \tr(e_2)$ is even.
We then have a well-defined map
\[
\begin{aligned}
\V^4(\Q(4)) & \to \V^2(\Q(4)) \oplus 2 \V^2(\Q(4)) \cong \N \oplus 2\N \\
[e] & \mapsto [e_1] \oplus [e_2 + e_3] \cong \tr(e_1) \oplus 2\, \tr(e_2);
\end{aligned}
\]
this is because the classes of $e_1$ and $e_1 + e_2$ are preserved by (stable)
similarity, and the $K_0$-class of an idempotent in a matrix ring over a
number
field (or a PID) is the rank (= trace). It is easy to check that this map is
injective
(using $e_1 \perp e_2 + e_3$ in $M_n(\Q(4))$) and additive. The only question is
surjectivity. It suffices to show that there is a $4$-potent $e$ over $\Q(4)$
whose stable similarity class is mapped to the generator $1 \oplus 2$ of $\N
\oplus 2\N$. Consider the block diagonal matrix
\[
e = \begin{pmatrix}
1 & 0 & 0 \\
0 & 0 & i \\
0 & i & -1 
\end{pmatrix} \in M_3(\Q(4)),
\]
which is easily checked to be quadripotent. The lower right quadripotent $2
\times 2$ {\it invertible} block has the desired eigenvalues $\omega$ and
$\bar{\omega}$, and so does not diagonalize over $\Q(4)$. The result now follows
easily.
\end{proof}

\section{$n$-Homomorphisms and $K_0^n$ Functorality}

We know from Proposition \ref{K0 is a functor} that $K_0^n$ is a covariant
functor from the category of (unital) rings and ring homomorphisms to
the category of abelian groups and group homomorphisms.  However, $K_0^n$ is
actually functorial for a more general class of ring mappings.

\begin{definition}\label{definition of n-homomorphism}
Let $R$ and $S$ be rings.   An additive map (not necessarily
unital)
$\phi: R \longrightarrow S$ is called an \emph{$n$-homomorphism} if
\[
\phi(a_1a_2 \cdots a_n) = \phi(a_1)\phi(a_2)\cdots \phi(a_n)
\]
for all $a_1, a_2, \dots, a_n$ in $R$. 
\end{definition}

Obviously every (ring) homomorphism is an $n$-homomorphism, but the converse is
false
in general. For example, an {\it $AE_n$-ring} is a ring $R$ such that every
additive
map $\phi : R \to R$ is an $n$-homomorphism. Feigelstock \cite{Feig1,
Feig2} classified all {\it unital} $AE_n$-rings. The algebraic version of
$n$-homomorphism was introduced for complex algebras in \cite{HMM} and has been
carefully  studied in the case of $C^*$-algebras in \cite{PT}.  

\begin{proposition}\label{n-functorality}
Let $\phi : R \to S$ be an $n$-homomorphism between unital rings.  Then $\phi$
induces a group homomorphism 
\[
\phi_*: K_0^n(R) \longrightarrow K_0^n(S).
\]  
Furthermore, the assignment $R \mapsto K_0^n(R)$ is a covariant functor from the
category of unital rings and $n$-homomorphisms to the category of abelian groups
and ordinary group homomorphisms.
\end{proposition}

\begin{proof}
For each natural number $k$, we extend $\phi$ to a map from $M_k(R)$ to $M_k(S)$
by applying $\phi$ to each matrix entry; it is easy to check this also gives us
an
$n$-homomorphism.  Moreover, $\phi$ is compatible with
stabilization of matrices; the only nonobvious point to check is that $\phi$
respects algebraic equivalence.

Let $e$ and $f$ be algebraically equivalent $n$-potents in $M_k(R)$ for some
$k$, and choose $a$ and $b$ in $M_k(R)$ so that $e = ab$ and $f = ba$.  Define
elements
$a^\prime = \phi(ea)\phi(f)^{n-2}$ and $b^\prime = \phi(b)$ in $M_k(S)$.  We
compute:
\begin{multline*}
a^\prime b^\prime =  \phi(ea)\phi(f)^{n-2}\phi(b) = 
\phi((ea)f^{n-2}b) = \\
\phi(ea(ba)^{n-2}b) = \phi(e(ab)^{n-1}) = \phi(e^n) = \phi(e).
\end{multline*}
A similar argument shows that $b^\prime a^\prime = \phi(f)$.  Therefore
$\phi$ determines a monoid homomorphism from $\V^n(R)$ to $\V^n(S)$, and hence a
group homomorphism $\phi_*: K_0^n(R) \longrightarrow K_0^n(S)$.  We leave it to
the reader to make the straightforward computations to show that we have a
covariant functor.
\end{proof}

Note that while we have an isomorphism $K_0^n(A) \cong
\bigl(K_0(A)\bigr)^{n-1}$ for $\Q(n-1)$-algebras, it is not at all clear from
the right hand side of this isomorphism that $K_0^n(A)$ is functorial for
$n$-homomorphisms.


\begin{thebibliography}{7}

\bibitem{Bla98}
B.~Blackadar, \emph{{$K$}-theory for {O}perator {A}lgebras}, 2nd ed., MSRI
  Publication Series 5, Springer-Verlag, New York, 1998.

\bibitem{Feig1}
S.~Feigelstock, \emph{Rings whose additive endomorphisms are
$N$-multiplicative}, Bull. Austral. Math. Soc. 39 (1989), no. 1, 11--14. 

\bibitem{Feig2}
S.~Feigelstock, \emph{Rings whose additive endomorphisms are $n$-multiplicative.
II}, Period. Math. Hungar. 25 (1992), no. 1, 21--26. 
  
\bibitem{HMM} M.~Hejazian, M.~Mirzavaziri, M.S.~Moslehian,
\emph{n-homomorphisms}, Bull. Iranian Math. Soc. 31 (2005), no. 1, 13-23.

\bibitem{PT} E.~Park and J.~Trout, \emph{On the Nonexistence of
Nontrivial Involutive $n$-homomorphisms of $C^*$-algebras}, Trans. Amer. Math. Soc. 361 (2009), no. 4, 1949--1961

\bibitem{RLL} M.~Rordam, F.~Larsen, N.~Laustsen, 
\emph{An Introduction to $K$-theory for $C^*$-algebras}, London Mathematical
Society Student Texts, vol. 49.  Cambridge University Press, Cambridge, 2000.

\bibitem{ROS} J.~Rosenberg, \emph{Algebraic {$K$}-theory and Its Applications},
Graduate Texts in Mathematics, vol. 147, Springer-Verlag, New York, 1994.

\end{thebibliography}
\end{document}